\documentclass[reqno,11pt]{amsart}
\usepackage[foot]{amsaddr}
\usepackage{bbold}
\usepackage{charter}
\usepackage[margin=1in]{geometry}
\usepackage[colorlinks=true,linktoc=all,linkcolor=blue,citecolor=blue]{hyperref}

\newtheorem{theorem}{Theorem}[section]
\newtheorem{corollary}[theorem]{Corollary}
\newtheorem{lemma}[theorem]{Lemma}

\theoremstyle{remark}
\newtheorem{remark}{Remark}
\newtheorem*{conjecture}{Conjecture}

\newcommand{\B}{\mathbb{B}}
\newcommand{\C}{\mathbb{C}}
\newcommand{\D}{\mathbb{D}}
\newcommand{\N}{\mathbb{N}}

\newcommand{\Bloch}{\mathcal{B}}
\newcommand{\lilBloch}{\mathcal{B}_0}
\newcommand{\norm}[1]{\left|\left|#1\right|\right|}
\newcommand{\supnorm}[1]{\norm{#1}_\infty}
\newcommand{\blochnorm}[1]{\norm{#1}_\Bloch}
\newcommand{\ip}[1]{\left\langle #1 \right\rangle}
\newcommand{\conj}[1]{\overline{#1}}
\newcommand{\Aut}{\mathrm{Aut}}

\renewcommand{\mod}[1]{\left|#1\right|}

\title[Weighted Composition Operators]
{Weighted Composition Operators from $H^\infty$ to the Bloch Space of a Bounded Homogeneous Domain}

\author{Robert F. Allen\textsuperscript{1} and Flavia Colonna\textsuperscript{2}}
\address{\textsuperscript{1}Department of Mathematics, University of Wisconsin--La Crosse}
\address{\textsuperscript{2}Department of Mathematical Sciences, George Mason University}
\email{allen.rob3@uwlax.edu, fcolonna@gmu.edu}

\date{}

\begin{document}

\begin{abstract}  Let $D$ be a bounded homogeneous domain in $\C^n$.  In this paper, we study the bounded and the compact weighted composition operators mapping the Hardy space $H^\infty(D)$ into the Bloch space of $D$. We characterize the bounded weighted composition operators, provide operator norm estimates, and give sufficient conditions for compactness. We prove that these conditions are necessary in the case of the unit ball and the polydisk. We then show that if $D$ is a bounded symmetric domain, the bounded multiplication operators from $H^\infty(D)$ to the Bloch space of $D$ are the operators whose symbol is bounded. \end{abstract}

\maketitle

\section{Introduction}
Let $X$ be a Banach space of holomorphic functions on a bounded domain $D$ in $\C^n$.  For $\psi$ a complex-valued holomorphic function on $D$ and $\varphi$ a holomorphic self-map of $D$, we define a linear operator $W_{\psi,\varphi}$ on $X$, called the \textit{weighted composition operator with multiplicative symbol $\psi$ and composition symbol $\varphi$}, by $$W_{\psi,\varphi} f = \psi(f\circ\varphi),\ f\in X.$$
Setting $M_\psi f=\psi f$ and $C_\varphi f=f\circ\varphi$, we may write $W_{\psi,\varphi}=M_\psi C_\varphi$. Then $M_\psi$ is called \textit{multiplication operator with symbol $\psi$} and $C_\varphi$ is called \textit{composition operator with symbol $\varphi$}.

	The study of the weighted composition operators on the Bloch space began with the work of Ohno and Zhao in \cite{OhnoZhao:01} where the operators from the Bloch space of the open unit disk $\D$ into itself were considered. In higher dimensions, these operators on the Bloch space have been studied by Chen, Stevi\'c and Zhou in \cite{ChenStevicZhou:09} (which was motivated by papers and \cite{Stevic:06} and \cite{ClahaneStevicZhou}), and by the authors in \cite{AllenColonna:09}. For related work, see \cite{ZhouChen:05}.

	In \cite{Ohno:01}, Ohno investigated the weighted composition operators between the Bloch space of $\D$ and the Hardy space $H^\infty(\D)$ of bounded analytic functions on $\D$.

	Characterizations of the boundedness and the compactness of the weighted composition operators from the Bloch space to $H^\infty$ were given by Hosokawa, Izuchi and Ohno in \cite{HosokawaIzuchiOhno:05} in the one-dimensional case, and by Li and Stevi\'c in the case of the unit ball \cite{LiStevic:08}. In \cite{Stevic:08}, Stevi\'c determined the norm of the bounded weighted composition operators from the Bloch space and the little Bloch space to the weighted Hardy space $H^\infty_\mu$ of the unit ball (where $\mu$ is a weight). The weighted composition operators from a larger class of spaces known as the $\alpha$-Bloch spaces to $H^\infty$ on the polydisk where studied by Li and Stevi\'c in \cite{LiStevic:07-II}. In \cite{AllenColonna:09}, we characterized the boundedness, determined the operator norm, and gave a sufficient condition for compactness in the case of a general bounded homogeneous domain in $\C^n$. %We also obtained independently a characterization of the compactness for the unit polydisk.

	The study of the weighted composition operators from $H^\infty$ to the Bloch space in several variables was carried out by Li and Stevi\'c when the ambient space is the unit polydisk \cite{LiStevic:07}. For the case of the unit ball, the study of these operators from $H^\infty$ into the $\alpha$-Bloch spaces was carried out by Li and Stevi\'c in \cite{LiStevic:08} and Zhang and Chen in \cite{ZhangChen:09}.

	In this paper, we analyze the weighted composition operators from $H^\infty$ into the Bloch space on a bounded homogeneous domain in $\C^n$.

	In Section 2, we give the background on the bounded homogeneous domains in $\C^n$ and special class of such domains that have a canonical representation (up to biholomorphic transformation) due to Cartan \cite{Cartan:35}, the bounded symmetric domains. We then review the notion of the Bloch space of a bounded homogeneous domain \cite{Hahn:75}, \cite{Timoney:80-I}, and of a subspace we refer to as the $*$-little Bloch space. We also recall the definition of little Bloch space on a bounded symmetric domain \cite{Timoney:80-II}.

In Section 3, in the environment of a general bounded homogeneous domain, we characterize the bounded weighted composition operators from $H^\infty$ into the Bloch space and into the $*$-little Bloch space, thereby extending the results of Ohno \cite{Ohno:01} in the one-dimensional case, of Li and Stevi\'c for the polydisk \cite{LiStevic:07} and for the unit ball \cite{LiStevic:08}, and of Zhang and Chen for the unit ball \cite{ZhangChen:09}. We also give estimates on the operator norm.

In Section 4, we describe sufficient conditions for the compactness of a weighted composition operator from $H^\infty$ to either the Bloch space or the $*$-little Bloch space of a bounded homogeneous domain. We conjecture these conditions to be necessary, and prove the necessity when the domain is the unit polydisk and the unit ball. In the latter setting, we obtain a result equivalent to a special case of Theorem 2 in \cite{ZhangChen:09}.  Furthermore, we show that compactness and boundedness are equivalent when the operator maps $H^\infty(\B_n)$ into $\lilBloch(\B_n)$, a result not observed in \cite{ZhangChen:09}.

In Section 5, we show that the bounded multiplication operators from $H^\infty$ into the Bloch space (respectively, the $*$-little Bloch space) of a bounded symmetric domain are precisely those whose symbol is bounded (respectively, in the $*$-little Bloch space and bounded). Furthermore, we obtain operator norm estimates in terms of the Bergman constant of the domain. We then discuss the boundedness of the composition operators on a bounded homogeneous domain  and establish operator norm estimates. While composition operators from $H^\infty$ to the $\alpha$-Bloch spaces of the polydisk have been studied by Stevi\'c in \cite{Stevic:06}, to our knowledge, there are no results in the literature for the multiplication or the composition operators between $H^\infty$ and the Bloch spaces on general bounded homogeneous or symmetric domains.

Finally, in Section 6, we show there are no isometries amongst the multiplication or composition operators from the Hardy space $H^\infty$ into the Bloch space when the domain is the unit disk. We also show that there are no isometric weighted composition operators from the Bloch space into $H^\infty$ if the ambient space is the unit polydisk. We conjecture that, likewise, there are no isometric weighted composition operators from $H^\infty$ to the Bloch space.

\section{Preliminaries}
Let $D$ be a domain in $\C^n$.  We denote by $H(D)$ the set of holomorphic functions from $D$ into $\C$, and by $\Aut(D)$ the set of biholomorphic maps of $D$. The space $H^\infty(D)$ of bounded holomorphic functions on $D$ is a Banach algebra equipped with norm $\supnorm{f} = \sup_{z \in D} \mod{f(z)}$.

A domain $D$ is {\it homogeneous} if $\Aut(D)$ acts transitively on $D$.  Every homogeneous domain is equipped with a canonical metric invariant under the action of $\Aut(D)$, called the {\it Bergman metric} \cite{Helgason:62}.  A domain $D$ is {\it symmetric at $z_0 \in D$} if there exists an involution $\varphi \in \Aut(D)$ for which $z_0$ is an isolated fixed point.  A domain that is symmetric at every point is called {\it symmetric}.  A bounded symmetric domain is homogeneous and a bounded homogeneous domain that is symmetric at one point is symmetric \cite{Helgason:62}. The unit ball
$$\B_n=\{z=(z_1,\dots,z_n)\in\C^n: \norm{z}^2=\sum_{k=1}^n |z_k|^2<1\}$$ and the unit polydisk $$\D^n=\{(z_1,\dots,z_n)\in\C^n: |z_k|<1, k=1,\dots,n\}$$ are bounded symmetric domains since they are homogeneous and symmetric at the origin via the map $z\mapsto -z$. While bounded homogeneous domains in dimensions 2 and 3 are symmetric, there are examples of bounded homogeneous domains in dimensions greater than 3 which are not symmetric \cite{Pjat:59}.

In \cite{Cartan:35}, Cartan showed that every bounded symmetric domain in $\C^n$ is biholomorphically equivalent to a finite product of irreducible bounded symmetric domains, unique up to order. He classified the irreducible bounded symmetric domains into six classes, four of which are known as the \emph{Cartan classical domains}, and the other two, each consisting of a single domain, are known as the \emph{exceptional domains}.  A bounded symmetric domain written as such a product is said to be in \emph{standard form}.

The Cartan classical domains are defined as:
$$\begin{aligned}
R_I &= \{Z \in M_{m,n}(\C) : I_m - ZZ^* > 0\}, \text{ for $m \geq n \geq 1$},\\
R_{II} &= \{Z \in M_n(\C) : Z=Z^T, I_n - ZZ^* > 0\}, \text{ for $n \geq 2$},\\
R_{III} &= \{Z \in M_n(\C) : Z = -Z^T, I_n - ZZ^* > 0\}, \text{ for $n \geq 5$},\\
R_{IV} &= \left\{z=(z_1,\dots,z_n) \in \C^n : A > 0, \norm{z} < 1\right\},\text{ for } n \geq 5,
\end{aligned}$$
where $M_{m,n}(\C)$ denotes the set of $m\times n$ matrices with complex entries, $M_n(\C) = M_{n,n}(\C)$, $Z^T$ and $Z^*$ are the transpose and the adjoint of $Z$, respectively, and $A = |\sum z_k^2|^2 + 1 - 2\norm{z}^2$.  See \cite{Drucker:78} for a description of the exceptional domains.

The notion of Bloch function in higher dimensions was first introduced by Hahn in \cite{Hahn:75}.  In \cite{Timoney:80-I}   and \cite{Timoney:80-II}, Timoney studied in depth the Bloch functions on a bounded homogeneous domain. In this paper, we conform to his definition and notation, as follows.

Let $D$ be a bounded homogeneous domain. For $z \in D$ and $f \in H(D)$, define $$Q_f(z) =
\sup_{u \in \C^n\setminus\{0\}} \frac{\mod{\nabla(f)(z)u}}{H_z(u,\conj{u})^{1/2}},$$ where $\nabla(f)(z)$ is the gradient of $f$ at $z$, for $u=(u_1,\dots,u_n)$, $$\nabla(f)(z)u = \sum_{k = 1}^n \frac{\partial f}{\partial z_k}(z)u_k,$$ and $H_z$ is the Bergman metric on $D$ at $z$.  The Bergman metric for the unit ball $\B_n$ is defined as
\begin{equation}%\label{bergman_Bn}
\nonumber
H_z(u,\conj{v}) = \frac{n+1}{2}\cdot\frac{(1-\norm{z}^2)\ip{u,v} + \ip{u,z}\ip{z,v}}{(1-\norm{z}^2)^2},
\end{equation} where $u,v \in \C^n, z \in \B_n$, and $\ip{u,v}=\sum u_j\overline{v_j}$.  For the unit polydisk $\D^n$, the Bergman metric is defined as
\begin{equation}%\label{bergman_Dn}
\nonumber
H_z(u,\conj{v}) = \sum_{j=1}^n \frac{u_j\conj{v_j}}{(1-\mod{z_j}^2)^2},
\end{equation} where $u,v \in \C^n$ and $z \in \D^n$.

The {\it Bloch space} $\Bloch(D)$ on a bounded homogeneous domain $D$ is the set of all functions $f \in H(D)$ for which $$\beta_f = \sup_{z \in D} Q_f(z) < \infty.$$  Timoney proved that $\Bloch(D)$ is a Banach space under the norm $$\blochnorm{f} = \mod{f(z_0)} + \beta_f,$$ where $z_0$ is some fixed point in $D$ \cite{Timoney:80-I}. For convenience, we shall assume throughout that $0 \in D$ and choose $z_0 = 0$.

The {\it $*$-little Bloch space of $D$} is the subspace of $\Bloch(D)$ defined as $$\Bloch_{0*}(D) = \left\{f \in \Bloch(D) : \lim_{z \to \partial^*D} Q_f(z) = 0\right\},$$ where $\partial^* D$ is the distinguished boundary of $D$. Timoney defined the {\it little Bloch space} $\lilBloch(D)$ of a bounded symmetric domain $D$ to be the closure of the polynomials in $\Bloch(D)$. Its elements $f$ satisfy the condition
$$\lim_{z\to\partial^*D}Q_f(z)=0.$$
If $D$ is the unit ball, then $\partial D = \partial^*D$ and thus $\lilBloch(D) = \Bloch_{0*}(D)$. When $D\neq \B_n$, $\lilBloch(D)$ is a proper subspace of $\Bloch_{0*}(D)$, and $\Bloch_{0*}(D)$ is a non-separable subspace of $\Bloch(D)$ \cite{Timoney:80-II}.

In \cite{Timoney:80-I}, Timoney proved that the space $H^\infty(D)$ of bounded holomorphic functions on a bounded homogeneous domain $D$ is a subspace of $\Bloch(D)$ and for each $f \in H^\infty(D)$, $\blochnorm{f} \leq |f(0)|+ c\supnorm{f}$ where $c$ is a constant depending only on the domain $D$.  %The precise value of the best bound $c_D$ has been calculated in \cite{CohenColonna:94} and \cite{Zhang:97} when $D$ is a bounded symmetric domain.

In \cite{CohenColonna:94}, Cohen and the second author defined the \emph{Bloch constant} of a bounded homogeneous domain $D$ in $\C^n$ as $$c_D = \sup\{Q_f(z) : f \in H^\infty(D), \norm{f}_\infty\leq 1, z\in D\}.$$  The precise value of the Bloch constant was calculated for each Cartan classical domain to be the reciprocal of the inner radius of the domain with respect to the Bergman metric, that is,
\begin{eqnarray}\nonumber c_D=\frac1{\inf_{\xi\in\partial D}H_0(\xi,\overline{\xi})^{1/2}}.\end{eqnarray}
 In \cite{Zhang:97}, Zhang determined the Bloch constant for the two exceptional domains $R_V$ of dimension 16 and $R_{VI}$ of dimension 27. The value of $c_D$ for each irreducible bounded symmetric domain $D$ is shown in Theorem~\ref{cd}.

By Theorem~3 of \cite{CohenColonna:94}, extended to include the exceptional domains, if $D = D_1\times\cdots\times D_k$ is a bounded symmetric domain in standard form, then
\begin{eqnarray} c_D = \max_{1\leq j\leq k} c_{D_j}.\label{cdformula}\end{eqnarray}
Furthermore, it was shown that there exist polynomial functions $f$ (hence in the little Bloch space of $D$) such that $\norm{f}_\infty\leq 1$ and $c_D=Q_f(0)$.

Let $D$ be a bounded homogeneous domain. In \cite{AllenColonna:07}, we showed that a function $f\in H(D)$ is Bloch if and only if there exists $c>0$ such that
\begin{eqnarray} |f(z)-f(w)|\leq c\rho(z,w), \hbox{ for each }z,w\in D,\nonumber\end{eqnarray}
where $\rho$ is the Bergman distance on $D$. Furthermore, the Bloch seminorm of $f$ is precisely the infimum of all such constants $c$. As a consequence, we obtain that for a Bloch function $f$ on $D$,
\begin{eqnarray} |f(z)-f(w)|\leq \norm{f}_\Bloch\rho(z,w), \hbox{ for each }z,w\in D.\label{Lipschitzcond}\end{eqnarray}

\section{Boundedness}
From now on, unless specified otherwise, we shall assume $D$ is a bounded homogeneous domain, $\psi \in H(D)$, and $\varphi=(\varphi_1,\dots,\varphi_n)$ is a holomorphic self-map of $D$.  Define
$\theta_{\psi,\varphi}=\sup_{z\in D}|\psi(z)|\theta_\varphi(z)$, where
$$\theta_\varphi(z)=\sup\{Q_{f\circ \varphi}(z): f\in H^\infty(D), \norm{f}_\infty\le 1\}.$$

For $z \in D$, denote by $J\varphi(z)$ the Jacobian matrix of $\varphi$ at $z$, that is, the matrix whose $(j,k)$-entry is $\frac{\partial\varphi_j}{\partial z_k}(z)$.
Furthermore, define $$B_\varphi(z) = \sup_{u \in \C^n\setminus\{0\}}\;\frac{H_{\varphi(z)}(J\varphi(z)u,\conj{J\varphi(z)u})^{1/2}}{H_z(u,\conj{u})^{1/2}}.$$
The quantity $B_\varphi = \sup_{z \in D}\; B_\varphi(z)$ is bounded above by a constant dependent only on $D$ \cite{Timoney:80-I}.
By the invariance of the Bergman metric under the action of $\hbox{Aut}(D)$, if $\varphi\in\hbox{Aut}(D)$, then for each $z\in D$, $u\in \C^n$,
$$H_{\varphi(z)}(J\varphi(z)u,\conj{J\varphi(z)u})=H_z(u,\conj{u}).$$
 For $f \in \Bloch(D)$ and $z \in D$, \begin{equation}\nonumber%\label{inequality:B_varphi}
Q_{f\circ\varphi}(z) \leq B_\varphi(z)Q_f(\varphi(z)).\end{equation}
Taking the supremum over all functions $f\in H^\infty(D)$ with $\supnorm{f}\leq 1$, we obtain
\begin{equation}\label{inequality:B_varphi}%\label{T-B inequality}
\theta_\varphi(z) \leq c_DB_\varphi(z)\end{equation} for each $z\in D$.

\begin{theorem}\label{ACbhd} {\normalfont{(a)}} The weighted composition operator $W_{\psi,\varphi}:H^\infty(D)\to\Bloch(D)$ is bounded if and only if $\psi\in\Bloch(D)$ and $\theta_{\psi,\varphi}$ is finite.
\smallskip

\noindent {\normalfont{(b)}} $W_{\psi,\varphi}:H^\infty(D)\to\Bloch_{0*}(D)$ is bounded if and only if $\psi\in\Bloch_{0*}(D)$, $\theta_{\psi,\varphi}$ is finite, and \begin{eqnarray}\lim_{z\to\partial^*D}|\psi(z)|\theta_\varphi(z)=0.\label{zero lim}\end{eqnarray}
Furthermore, if $W_{\psi,\varphi}$ is bounded as an operator into $\Bloch(D)$ or $\Bloch_{0*}(D)$, then the norm of $W_{\psi,\varphi}$ satisfies the estimates
\begin{eqnarray} \norm{\psi}_\Bloch \leq \norm{W_{\psi,\varphi}}\leq \norm{\psi}_\Bloch+\theta_{\psi,\varphi}.\label{estim}\end{eqnarray}
\end{theorem}

\begin{proof} (a) Assume first $W_{\psi,\varphi}$ is bounded. Then $\norm{\psi}_\Bloch=\norm{W_{\psi,\varphi}1}_\Bloch$, thus $\psi\in\Bloch(D)$ and the lower estimate of (\ref{estim}) holds. For each $f\in H^\infty(D)$ with $\norm{f}_\infty\le 1$, and each $z\in D$, we have
\begin{eqnarray} \nabla(\psi(f\circ\varphi))(z) = \psi(z)\nabla(f\circ\varphi)(z) + f(\varphi(z))\nabla(\psi)(z)\label{nabla-rule}\end{eqnarray} so that
$$\mod{\psi(z)}Q_{f\circ \varphi}(z)\le \norm{W_{\psi,\varphi}f}_\Bloch +Q_\psi(z)\leq \norm{W_{\psi,\varphi}}+\norm{\psi}_\Bloch.$$
Thus, taking the supremum over all such functions $f$ and over all $z\in D$, we see that $\theta_{\psi,\varphi}$ is finite.

Conversely, assume $\psi\in\Bloch(D)$ and $\theta_{\psi,\varphi}$ is finite. Then for $f\in H^{\infty}(D)$, with $\|f\|_\infty\leq 1$, we have
\begin{eqnarray} Q_{\psi(f\circ\varphi)}(z)\leq |\psi(z)|Q_{f\circ\varphi}(z)+|f(\varphi(z))|Q_\psi(z)\nonumber\leq \theta_{\psi,\varphi}+\beta_\psi.\nonumber\end{eqnarray}
Thus, $W_{\psi,\varphi}f\in \Bloch(D)$ and
$$\norm{\psi(f\circ\varphi)}_\Bloch\leq |\psi(0)f(\varphi(0))|+\theta_{\psi,\varphi}+\beta_\psi\leq \norm{\psi}_\Bloch+\theta_{\psi,\varphi},$$
proving the boundedness of $W_{\psi,\varphi}$ and the upper estimate of (\ref{estim}).\\
\\
\noindent (b) From part (a), it suffices to show that if $W_{\psi,\varphi}$ is bounded, then (\ref{zero lim}) holds and, conversely, if $\psi\in\Bloch_{0*}(D)$, $\theta_{\psi,\varphi}$ is finite, and (\ref{zero lim}) holds, then for each $f\in \Bloch_{0*}(D)$, $\psi(f\circ \varphi)\in \Bloch_{0*}(D).$

Assume $W_{\psi,\varphi}$ is bounded. Then $\psi = W_{\psi,\varphi}1 \in \Bloch_{0*}(D)$ and, for each $f\in H^\infty(D)$ and each $z\in D$, from (\ref{nabla-rule}) we have %$$\nabla(\psi(f\circ\varphi))(z) = \psi(z)\nabla(f\circ\varphi)(z) + f(\varphi(z))\nabla(\psi)(z)$$ so that
$$|\psi(z)|Q_{f\circ \varphi}(z)\leq Q_{\psi(f\circ \varphi)}(z)+Q_\psi(z)\norm{f}_\infty.$$
Taking the limit as $z\to\partial^*D$, we obtain $|\psi(z)|Q_{f\circ \varphi}(z)\to 0$. Since this holds for each $f\in H^\infty(D)$, we deduce (\ref{zero lim}).

Next assume $\psi\in\Bloch_{0*}(D)$, $\theta_{\psi,\varphi}$ is finite, and (\ref{zero lim}) holds. Then for each nonzero function $f\in H^\infty(D)$ and each $z\in D$, using the function $\tilde{f}=f/\norm{f}_\infty$, we see that $Q_{f\circ \varphi}(z)\le \theta_\varphi(z)\norm{f}_\infty$. Thus, for each $f\in H^\infty(D)$ and each $z\in D$, we obtain
\begin{eqnarray}Q_{\psi(f\circ \varphi)}(z)&\leq& Q_\psi(z)\|f||_\infty+|\psi(z)|Q_{f\circ\varphi}(z)\nonumber\\
&\leq& (Q_\psi(z)+|\psi(z)|\theta_{\varphi}(z))\norm{f}_\infty.\nonumber\end{eqnarray} Taking the limit as $z\to\partial^*D$, we conclude that $\psi(f\circ \varphi)\in \Bloch_{0*}(D)$.
\end{proof}

In the special case of $D=\B_n$, the finiteness of $\theta_{\psi,\varphi}$ follows from (\ref{zero lim}) since $\partial^*\B_n=\partial B_n$.  From this, we obtain the following characterization of the bounded weighted composition operators from $H^\infty(\B_n)$ into $\lilBloch(\B_n)$.

\begin{corollary}\label{corollary:bounded_on_B_n} $W_{\psi,\varphi}:H^\infty(\B_n) \to \lilBloch(\B_n)$ is bounded if and only if $\psi \in \lilBloch(\B_n)$ and \begin{equation} \lim_{\norm{z} \to 1} \mod{\psi(z)}\theta_\varphi(z) = 0.\nonumber\end{equation}\end{corollary}

In the case of $D = \D$, we can improve the lower bound on the operator norm of $W_{\psi,\varphi}$.

\begin{lemma} Let $\psi \in H(\D)$ and $\varphi$ a holomorphic self-map of $\D$.  Then
\begin{equation}\label{theta on disk}\theta_{\psi,\varphi} = \sup_{z \in \D}\mod{\psi(z)}\frac{(1-\mod{z}^2)\mod{\varphi'(z)}}{1-\mod{\varphi(z)}^2}.\end{equation}
\end{lemma}

\begin{proof} Let $f \in H^\infty(\D)$ such that $\supnorm{f} \leq 1$. Since $\beta_f\leq \supnorm{f}$, for all $z \in \D$,
$$\begin{aligned}
Q_{f\circ\varphi}(z) &= (1-\mod{z}^2)\mod{f'(\varphi(z))}\mod{\varphi'(z)}\\
&= (1-\mod{\varphi(z)}^2)\mod{f'(\varphi(z))}\frac{(1-\mod{z}^2)\mod{\varphi'(z)}}{1-\mod{\varphi(z)}^2}\\
&\leq \beta_f\frac{(1-\mod{z}^2)\mod{\varphi'(z)}}{1-\mod{\varphi(z)}^2}\\
&\leq \frac{(1-\mod{z}^2)\mod{\varphi'(z)}}{1-\mod{\varphi(z)}^2}.
\end{aligned}$$  Thus $$\theta_{\varphi}(z) \leq \frac{(1-\mod{z}^2)\mod{\varphi'(z)}}{1-\mod{\varphi(z)}^2}.$$  On the other hand, fixing $z \in \D$, the function $$f(w) = \frac{\varphi(z)-w}{1-\conj{\varphi(z)}w},\ w\in\D,$$ is an automorphism of $\D$, and thus $\supnorm{f} = 1$. Moreover, \begin{eqnarray}Q_{f\circ\varphi}(z) = \frac{(1-\mod{z}^2)\mod{\varphi'(z)}}{1-\mod{\varphi(z)}^2}\label{hypderivative}\end{eqnarray} and so $\theta_\varphi(z) = \frac{(1-\mod{z}^2)\mod{\varphi'(z)}}{1-\mod{\varphi(z)}^2},$  which yields (\ref{theta on disk}).
\end{proof}

\begin{theorem}\label{betterestimates} Let $W_{\psi,\varphi}$ be bounded from $H^\infty(\D)$ to $\Bloch(\D)$ or $\lilBloch(\D)$.  Then
$$\max\{\blochnorm{\psi}, \theta_{\psi,\varphi}\} \leq \norm{W_{\psi,\varphi}} \leq \blochnorm{\psi} + \theta_{\psi,\varphi}.$$
\end{theorem}

\begin{proof} By Theorem~\ref{ACbhd}, It suffices to show that $\theta_{\psi,\varphi} \leq \norm{W_{\psi,\varphi}}$.  Fix $\lambda \in \D$ and for $z\in\D$,
let $$f_\lambda(z) = \frac{\varphi(\lambda)-z}{1-\conj{\varphi(\lambda)}z}.$$
Then $\supnorm{f_\lambda} = 1$ and $f_\lambda(\varphi(\lambda))=0$. From (\ref{hypderivative}), we obtain
$$\begin{aligned}
\norm{W_{\psi,\varphi}} &\geq \blochnorm{W_{\psi,\varphi} f_\lambda}\\
&\geq \sup_{z \in \D} (1-\mod{z}^2)\mod{\psi'(z)f_\lambda(\varphi(z)) + \psi(z)(f_\lambda\circ\varphi)'(z)}\\
&\geq |\psi(\lambda)|Q_{f_\lambda\circ\varphi}(\lambda)\\
&= \frac{(1-\mod{\lambda}^2)\mod{\psi(\lambda)}\mod{\varphi'(\lambda)}}{1-\mod{\varphi(\lambda)}^2}.
\end{aligned}$$  Taking the supremum over all $\lambda \in \D$, we get $\norm{W_{\psi,\varphi}}\geq \theta_{\psi,\varphi}$.\end{proof}

\section{Compactness}
The following lemma will be used to prove a sufficient condition for the compactness of $W_{\psi,\varphi}$.

\begin{lemma}\label{lemma: Compactness of wco on H^infty}  $W_{\psi,\varphi}$ is compact from $H^\infty(D)$ into $\Bloch(D)$ if and only if for every bounded sequence $\{f_k\}$ in $H^\infty(D)$ converging to 0 locally uniformly in $D$, $\blochnorm{\psi(f_k\circ\varphi)} \to 0$ as $k \to \infty$.\end{lemma}

\begin{proof} Assume that $W_{\psi,\varphi}:H^\infty(D)\to \Bloch(D)$ is compact.  Let $\{f_k\}$ be a bounded sequence in $H^\infty(D)$ which converges to 0 locally uniformly in $D$.  By rescaling $f_k$, we may assume $\supnorm{f_k} \leq 1$ for all $k \in \N$.  We need to show that $\blochnorm{\psi(f_k\circ\varphi)} \to 0$ as $k \to \infty$.  Since $W_{\psi,\varphi}$ is compact, the sequence $\{\psi(f_k\circ\varphi)\}$ has a subsequence (which for convenience we reindex as the original sequence) converging in the Bloch norm to some function $f \in \Bloch(D)$.  We are going to show that $f$ is identically 0 by proving that $\psi(f_k\circ\varphi) \to 0$ locally uniformly.  Fix $z_0 \in D$ and, subtracting from the elements of this sequence the value of $f$ at $z_0$, we may assume $f(z_0) = 0$.  Then $\psi(z_0)f_k(\varphi(z_0)) \to 0$ as $k \to \infty$.  For $z \in D$, using (\ref{Lipschitzcond}), we find
$$\begin{aligned}
\mod{\psi(z)f_k(\varphi(z)) - f(z)} &\leq \mod{\psi(z)f_k(\varphi(z)) - f(z) - (\psi(z_0)f_k(\varphi(z_0)) - f(z_0))}\\
 &\qquad + \mod{\psi(z_0)f_k(\varphi(z_0))}\\
&\leq \blochnorm{\psi(f_k\circ\varphi) - f}\rho(z,z_0) + \mod{\psi(z_0)f_k(\varphi(z_0))}.\end{aligned}$$ The right-hand side converges to 0 locally uniformly as $k \to \infty$ since $\psi(f_k\circ\varphi)-f \to 0$ in $\Bloch(D)$.  On the other hand, $\psi(f_k\circ\varphi) \to 0$ locally uniformly, so $f$ is identically 0.

	Conversely, suppose that whenever $\{g_k\}$ is a bounded sequence in $H^\infty(D)$  converging to 0 locally uniformly in $D$, $\blochnorm{\psi(g_k\circ\varphi)} \to 0$ as $k \to \infty$. To prove the compactness of $W_{\psi,\varphi}$, it suffices to show that if $\{f_k\}$ is a sequence in $H^\infty(D)$ with $\supnorm{f_k} \leq 1$ for all $k \in \N$, there exists a subsequence $\{f_{k_j}\}$ such that $\psi(f_{k_j}\circ\varphi)$ converges in norm in $\Bloch(D)$.  Fix $z_0 \in D$.  Replacing $f_k$ by $f_k-f_k(z_0)$, we may assume that $f_k(z_0) = 0$ for all $k \in \N$.  Since $\{f_k\}$ is uniformly bounded on $D$, by Montel's theorem some subsequence $\{f_{k_j}\}$ converges locally uniformly to some holomorphic function $f$ on $D$ such that $\supnorm{f} \leq 1$. Then, letting $g_{k_j} = f_{k_j} - f$, we obtain a bounded sequence in $H^\infty(D)$ converging to 0 locally uniformly in $D$.  By the hypothesis, $\blochnorm{\psi(g_{k_j}\circ\varphi)} \to 0$ as $j \to \infty$.  Therefore, $\psi(f_{k_j}\circ\varphi)$ converges in norm to $\psi(f\circ\varphi)$, completing the proof.
\end{proof}

\begin{theorem}\label{suffcondcomp} {\normalfont{(a)}} Let $\psi \in \Bloch(D)$.  Then $W_{\psi,\varphi}:H^\infty(D)\to\Bloch(D)$ is compact if
\begin{equation}\label{compactness_on_H^infty}\lim_{\varphi(z) \to\partial D}Q_\psi(z) =0, \hbox{ and } \lim_{\varphi(z)\to\partial D} \mod{\psi(z)}\theta_\varphi(z) = 0.\end{equation}

\noindent {\normalfont{(b)}} Let $\psi \in \Bloch_{0*}(D)$. Then $W_{\psi,\varphi}:H^\infty(D) \to \Bloch_{0*}(D)$ is compact if \begin{equation} \lim_{z\to\partial D} Q_\psi(z) =0, \hbox{ and } \lim_{z\to\partial D}\mod{\psi(z)}\theta_\varphi(z) = 0.\label{cond*bloch}\end{equation}
\end{theorem}

\begin{proof} Assume conditions (\ref{compactness_on_H^infty}) both hold.  %Then $\theta_{\psi,\varphi}$ is finite so that $W_{\psi,\varphi}$ is bounded by Theorem \ref{ACbhd}.
 By Lemma \ref{lemma: Compactness of wco on H^infty}, to prove that $W_{\psi,\varphi}$ is compact from $H^\infty(D)$ into $\Bloch(D)$ it suffices to show that for any bounded sequence $\{f_k\}$ such that $\supnorm{f_k} \leq 1$ and $f_k \to 0$ locally uniformly in $D$, $\blochnorm{\psi(f_k\circ\varphi)} \to 0$ as $k \to \infty$.  Let $\{f_k\}$ be such a sequence and fix $\varepsilon > 0$.  Then, there exists $r > 0$ such that for all $k \in \N$, $\mod{\psi(z)}Q_{f_k\circ\varphi}(z) < \frac{\varepsilon}{2}$ and $Q_\psi(z) < \frac{\varepsilon}{2}$ whenever $\rho(\varphi(z),\partial D) \leq r$.  Thus, if $\rho(\varphi(z),\partial D) \leq r$, then $$Q_{\psi(f_k\circ\varphi)}(z) \leq \mod{\psi(z)}Q_{f_k\circ\varphi}(z) + Q_{\psi}(z) < \varepsilon.$$  On the other hand, since $f_k \to 0$ locally uniformly in $D$, $\mod{f_k(\varphi(z))} \to 0$ uniformly on the compact set $E_r=\{z \in D : \rho(\varphi(z),\partial D) \geq r\}$. Thus, $\nabla f_k(\varphi(z))$ approaches the zero vector, and hence $Q_{f_k\circ\varphi}(z) \to 0$, uniformly on $E_r$ as $k\to\infty$.  Consequently, recalling that $\psi\in\Bloch(D)$, we see that for all $k$ sufficiently large, $Q_{\psi(f_k\circ\varphi)}(z) < \varepsilon$ for all $z \in D$.  Furthermore, $\mod{\psi(0)f_k(\varphi(0))} \to 0$ as $k \to \infty$, so $\blochnorm{\psi(f_k\circ\varphi)} \to 0$, completing the proof of (a).

The proof of part (b) is analogous.
\end{proof}

\begin{remark} If $D$ is not the unit ball, the multiplicative symbol of the weighted composition operators satisfying conditions (\ref{cond*bloch}) reduces
to a constant, and hence the compact weighted composition operators of this type have composition component which is compact. The case when $D=\B_n$ is discussed in Corollary~\ref{charcomplilBloch}.
\end{remark}

We conjecture that under boundedness assumptions, conditions (\ref{compactness_on_H^infty}) and (\ref{cond*bloch}) are necessary as well.

\begin{conjecture} If $D$ is a bounded homogeneous domain, then the bounded operator $W_{\psi,\varphi}:H^\infty(D)\to\Bloch(D)$ is compact if and only if
$$\lim_{\varphi(z) \to \partial D}Q_\psi(z) = 0, \hbox{ and }\lim_{\varphi(z)\to\partial D} \mod{\psi(z)}\theta_\varphi(z) = 0.$$
\end{conjecture}

\subsection{Compactness from $H^\infty(\B_n)$ to $\lilBloch(\B_n)$ or $\Bloch(\B_n)$}
	We begin this section by extending Theorem 4 in \cite{Ohno:01} to the unit ball.

\begin{corollary}\label{charcomplilBloch} Let $\varphi$ be a holomorphic self-map of $\B_n$ and $\psi \in H(\B_n)$.  Then the following are equivalent:
\begin{enumerate}
\item[(a)] $W_{\psi,\varphi}:H^\infty(\B_n) \to \lilBloch(\B_n)$ is bounded.
\item[(b)] $W_{\psi,\varphi}:H^\infty(\B_n) \to \lilBloch(\B_n)$ is compact.
\item[(c)] $\psi \in \lilBloch(\B_n)$ and $\displaystyle\lim_{\norm{z} \to 1}\mod{\psi(z)}\theta_\varphi(z) = 0.$
\end{enumerate}\end{corollary}

\begin{proof}  The implication $(b) \Longrightarrow (a)$ is obvious.  The implication $(a) \Longrightarrow (c)$ follows from Corollary \ref{corollary:bounded_on_B_n}.  Finally, $(c) \Longrightarrow (b)$ follows from part (b) of Theorem~\ref{suffcondcomp}.
\end{proof}

	We now prove the above conjecture in the case of the unit ball. The following result is equivalent to a characterization of the compactness obtained by Zhang and Chen (\cite{ZhangChen:09}, Theorem~2).

\begin{theorem}\label{compactnessball} Let $\varphi$ be a holomorphic self-map of $\B_n$, and $\psi\in H(\B_n)$. Then $W_{\psi,\varphi}:H^\infty(\B_n) \to\Bloch(\B_n)$ is compact if and only if it is bounded,
\begin{eqnarray}&\lim\limits_{\norm{\varphi(z)}\to 1} Q_\psi(z) = 0, \hbox{ and }\label{c_1}\\
&\lim\limits_{\norm{\varphi(z)} \to 1}\mod{\psi(z)}\theta_\varphi(z) = 0.\label{c_2}\end{eqnarray}
\end{theorem}

\begin{proof} If $W_{\psi,\varphi}$ is bounded and conditions (\ref{c_1}) and (\ref{c_2}) hold, then $\psi\in\Bloch(\B_n)$, so by Theorem~\ref{suffcondcomp}, $W_{\psi,\varphi}$ is compact.

Conversely, suppose $W_{\psi,\varphi}$ is compact. Then $W_{\psi,\varphi}$ is bounded and by Theorem~2 in \cite{ZhangChen:09}, condition (\ref{c_1}) holds, and
$$\lim_{\norm{\varphi(z)}\to 1}\mod{\psi(z)}B_\varphi(z)=0.$$
Condition (\ref{c_2}) now follows from the inequality $\theta_\varphi(z)\leq B_\varphi(z)$ for each $z\in\B_n$.
\end{proof}

\subsection{Compactness from $H^\infty(\D^n)$ to $\Bloch(\D^n)$}
In \cite{LiStevic:07}, Li and Stevi\'{c} characterized the compact weighted composition operators from $H^\infty(\D^n)$ into $\Bloch(\D^n)$ in the following result.

\begin{theorem}[\cite{LiStevic:07}, Theorem 1.2]\label{LiStevisPoly} Let $\varphi = (\varphi_1,\dots,\varphi_n)$ be a holomorphic self-map of $\D^n$ and $\psi(z)$ a holomorphic function on $\D^n$.  Then $W_{\psi,\varphi}:H^\infty(\D^n) \to \Bloch(\D^n)$ is compact if and only if the following conditions are satisfied:
\begin{enumerate}
\item[(a)] $W_{\psi,\varphi}:H^\infty(\D^n) \to \Bloch(\D^n)$ is bounded;
\item[(b)] $\displaystyle\lim_{\varphi(z) \to \partial\D^n}\sum_{k=1}^n (1-\mod{z_k}^2)\mod{\frac{\partial\psi}{\partial z_k}(z)} = 0$;
\item[(c)] $\displaystyle\lim_{\varphi(z) \to \partial\D^n}\mod{\psi(z)}\sum_{k,j = 1}^n\mod{\frac{\partial\varphi_j}{\partial z_k}(z)}\frac{1-\mod{z_k}^2}{1-\mod{\varphi_j(z)}^2} = 0.$
\end{enumerate}
\end{theorem}

We will prove the conjecture posed in the previous section in the setting of the polydisk $\D^n$.  To do this, we need the following lemma.

\begin{lemma}\label{theta_inequality} Let $\varphi = (\varphi_1,\dots,\varphi_n)$ be a holomorphic self-map of $\D^n$.  For $z \in \D^n$, $$\mod{\psi(z)}\theta_\varphi(z) \leq \mod{\psi(z)}\sum_{j,k = 1}^n \mod{\frac{\partial\varphi_j}{\partial z_k}(z)}\frac{1-\mod{z_k}^2}{1-\mod{\varphi_j(z)}^2}.$$\end{lemma}

\begin{proof} Observe that by (1.2) of \cite{CohenColonna:08}, for all $z \in \D^n$,
$$\begin{aligned}\mod{\psi(z)}\theta_\varphi(z) &\leq \mod{\psi(z)}B_\varphi(z)\\
& = \mod{\psi(z)}\max_{\norm{w}=1}\left(\sum_{j=1}^n\left|\sum_{k=1}^n\frac{\partial \varphi_j}{\partial z_k}(z)\frac{(1-\mod{z_k}^2)w_k}{1-\mod{\varphi_j(z)}^2}\right|^2\right)^{1/2}\\
&\leq \mod{\psi(z)}\max_{\norm{w}=1}\left(\sum_{j=1}^n\left(\sum_{k=1}^n\left|\frac{\partial \varphi_j}{\partial z_k}(z)\right|\frac{(1-\mod{z_k}^2)\mod{w_k}}{1-\mod{\varphi_j(z)}^2}\right)^2\right)^{1/2}\\
&\leq \mod{\psi(z)}\sum_{k,j=1}^n\left|\frac{\partial \varphi_j}{\partial z_k}(z)\right|\frac{1-\mod{z_k}^2}{1-\mod{\varphi_j(z)}^2}.\qedhere\end{aligned}$$
\end{proof}

\begin{theorem}\label{compacnesspolydisk} Let $\varphi$ be a holomorphic self-map of
$\D^n$ and $\psi \in \Bloch(\D^n)$.  Then $W_{\psi,\varphi}:H^\infty(\D^n)\to\Bloch(\D^n)$ is compact if and only if
\begin{eqnarray} \lim_{\varphi(z) \to \partial \D^n}Q_\psi(z) =0 \hbox{ and } \lim_{\varphi(z)\to\partial \D^n} \mod{\psi(z)}\theta_\varphi(z) = 0.\label{condipoli}\end{eqnarray}
\end{theorem}

\begin{proof} By Theorem \ref{suffcondcomp}, it suffices to show that if $W_{\psi,\varphi}$ is compact from $H^\infty(\D^n)$ into $\Bloch(\D^n)$, then conditions (\ref{condipoli}) hold.  First, observe that by Theorem 3.3 of \cite{CohenColonna:08}, for $z \in \D^n$, \begin{eqnarray} Q_\psi(z) &=&\norm{\left((1-|z_1|^2)\frac{\partial \psi}{\partial z_1}(z),\dots,(1-|z_n|^2)\frac{\partial \psi}{\partial z_n}(z)\right)}\nonumber\\&\leq &\sum_{k=1}^n (1-\mod{z_k}^2)\mod{\frac{\partial\psi}{\partial z_k}(z)}.\nonumber\end{eqnarray}  Since $W_{\psi,\varphi}$ is compact, Theorem \ref{LiStevisPoly}(b) implies that $$\lim_{\varphi(z) \to \partial\D^n} Q_\psi(z) = 0.$$

In addition, by Lemma \ref{theta_inequality}, we have
$$\mod{\psi(z)}\theta_\varphi(z) \leq \mod{\psi(z)}\sum_{k,j = 1}^n\mod{\frac{\partial\varphi_j}{\partial z_k}(z)}\frac{1-\mod{z_k}^2}{1-\mod{\varphi_j(z)}^2}$$ for all $z \in \D^n$.  Thus, by part (c) of Theorem~\ref{LiStevisPoly}, we obtain $$\lim_{\varphi(z) \to \partial\D^n} \mod{\psi(z)}\theta_\varphi(z) = 0,$$ completing the proof.
\end{proof}

\section{Component Operators}
In this section, we look at the issues of boundedness and compactness of the multiplication and the composition operators separately.

\subsection{Multiplication Operators from $H^\infty$ into the Bloch Space}
Let us now consider the implications of Theorem~\ref{ACbhd} for the case when $\varphi$ is the identity and $D$ is a bounded symmetric domain.

\begin{theorem}\label{multchar} Let $D$ be a bounded symmetric domain in standard form and let $\psi\in H(D)$. Then
\begin{enumerate}
\item[(a)] $M_{\psi}:H^\infty(D)\to\Bloch(D)$ is bounded if and only if $\psi\in H^\infty(D)$.
\item[(b)] $M_{\psi}:H^\infty(D)\to\Bloch_{0*}(D)$ is bounded if and only if $\psi\in\Bloch_{0*}(D)\cap H^\infty(D)$.\end{enumerate}
Furthermore, if $M_{\psi}$ is bounded as an operator from $H^\infty(D)$ into the Bloch space or the $*$-little Bloch space, then
\begin{eqnarray} \max\{\blochnorm{\psi},c_D\supnorm{\psi}\} \leq \norm{M_{\psi}}\leq \blochnorm{\psi}+c_D\supnorm{\psi}, \nonumber\end{eqnarray}
where $c_D$ is the Bloch constant of $D$.
\end{theorem}

\begin{proof}
Suppose first $D$ is an irreducible domain. In \cite{CohenColonna:94} and \cite{Zhang:97}, it was shown that there exists a polynomial $p$ on $D$ and $\xi\in\partial D$ such that $p(0)=0$, $\supnorm{p}=1$, $|\nabla{p}(0)\xi|=1$, and \begin{eqnarray}Q_p(0)=\frac1{H_0(\xi,\overline{\xi})^{1/2}}=c_D.\label{qzero}\end{eqnarray} If $D=D_1\times \dots \times D_k$, a product of irreducible domains, then by (\ref{cdformula}), there exists $j=1,\dots,k$ such that $c_D=c_{D_j}$, and so there exists a polynomial $p_j$ on $D_j$ and $\xi_j\in\partial D_j$ such that $\supnorm{p_j}=1$, $|\nabla(p_j)(0)\xi_j|=1$, and $Q_{p_j}(0)=c_D$. Then letting $p$ be the polynomial on $D$ such that $p(z)=p_j(z_j)$ (where $z_j$ denotes the component of $z$ in $D_j$), and  $\xi$ the vector whose component in $D_j$ equals $\xi_j$ and whose components in each irreducible factor other than $D_j$ are 0, we obtain a function on $D$ with supremum norm 1 and a vector in $\partial D$ such that $|\nabla(p)(0)\xi|=|\nabla(p_j)(0)\xi_j|=1$, $p(0)=0$, and $Q_p(0)=Q_{p_j}(0)=c_D$.

Fix $a\in D$ and let $S\in\hbox{Aut}(D)$ be such that $S(a)=0$. Since the Jacobian matrix $JS(a)$ is invertible, there exists a nonzero $v\in\C^n$ such that $JS(a)v=\xi$. Composing $p$ with $S$, we obtain a function $g\in \Bloch_{0*}(D)$ such that $\supnorm{g}=1$ and $Q_g(a)=c_D$. In particular, $$\theta_{\psi,\mathrm{id}}=\sup_{a\in D}|\psi(a)|\sup\{Q_f(a):\norm
{f}_\infty\le 1\}=c_D\norm{\psi}_\infty.$$ Moreover, by the boundedness of $M_\psi$, the invariance of the Bergman metric under the action of $\hbox{Aut}(D)$, recalling that $|\nabla p(0)\xi|=1$ and using (\ref{qzero}), we obtain
\begin{eqnarray}\norm{M_\psi g}_\Bloch&\geq & Q_{\psi g}(a)= \sup_{u\in\C^n\backslash\{0\}}\frac{|\nabla(\psi g)(a)u|}{H_a(u,\overline{u})^{1/2}}\nonumber\\
&=& \sup_{u\in\C^n\backslash\{0\}}\frac{|\psi(a)\nabla(p)(0)JS(a)u|}{H_{0}(JS(a)u,\overline{JS(a)u})^{1/2}}\nonumber\\
&\geq & |\psi(a)|\frac{|\nabla(p)(0)\xi|}{H_0(\xi,\overline{\xi})^{1/2}}=|\psi(a)|Q_p(0)\nonumber\\&=& |\psi(a)|c_D.\nonumber\end{eqnarray}
Taking the supremum over all $a\in D$, we deduce $\norm{M_\psi}\geq c_D\supnorm{\psi}.$ The result now follows at once from Theorem~\ref{ACbhd}.\end{proof}

\begin{theorem}[\cite{CohenColonna:94},\cite{Zhang:97}]\label{cd} If $D$ is an irreducible bounded symmetric domain in $\C^n$, then
\begin{eqnarray}\nonumber c_D=\begin{cases} \sqrt{\frac{2}{m+n}} &\hbox{ if }D\hbox{ is of type }R_I,\\
\sqrt{\frac{2}{n+1}} &\hbox{ if }D \hbox{ is of type }R_{II},\\
\frac{1}{\sqrt{n-1}} &\hbox{ if }D \hbox{ is of type }R_{III},\\
\sqrt{\frac{2}{n}} &\hbox{ if }D \hbox{ is of type }R_{IV},\\
\frac{1}{\sqrt{6}} &\hbox{ if }D =R_{V},\\
\frac{1}{3} &\hbox{ if }D =R_{VI}.\end{cases}\end{eqnarray}
\end{theorem}

	Recalling that the Bloch seminorm of a bounded analytic function is no greater than its supremum norm and observing from Theorem~\ref{cd} that $c_D\le 1$ for any bounded symmetric domain $D$ in standard form, and the only domains $D$ for which $c_D=1$ are those which contain the unit disk as a factor, we obtain the following result.

\begin{corollary} Let $D$ be a bounded symmetric domain in standard form with $\D$ as a factor and let $\psi\in H(D)$. If $M_\psi$ is bounded from $H^\infty(D)$ into either $\Bloch(D)$ or $\Bloch(D)_{0*}$, then $$\max\{\blochnorm{\psi},\supnorm{\psi}\} \leq \norm{M_\psi} \leq \blochnorm{\psi} + \supnorm{\psi}.$$  Furthermore, if $\psi(0) = 0$, then $$\supnorm{\psi} \leq \norm{M_\psi} \leq \blochnorm{\psi} + \supnorm{\psi}.$$
\end{corollary}

	We now characterize the compact multiplication operators from $H^\infty$ to $\Bloch$ when the underlying space is the ball or the polydisk.

\begin{theorem}\label{multcompactness} For $D=\B_n$ or $\D^n$, the following statements are equivalent:
\begin{enumerate}
\item[(a)] $M_{\psi}:H^\infty(D)\to\Bloch(D)$ is compact.
\item[(b)] $\psi$ is identically zero.\end{enumerate}
 \end{theorem}

\begin{proof} $(a)\Longrightarrow (b)$: %Assume first $D=\D^n$.
By Theorems~\ref{compactnessball} and \ref{compacnesspolydisk}, the compactness of $M_\psi$ implies that %$\lim_{z\to\partial D^n}Q_\psi(z)=0$, which implies that $\psi$ is a constant function. Since
$$\lim_{z\to\partial D}|\psi(z)|\sup\{Q_f(z): \supnorm{f}\leq 1\}=0.$$
Since $\sup\{Q_f(z): \supnorm{f}\leq 1\}=c_D$, it follows that $\lim\limits_{z\to\partial D}|\psi(z)|=0$, hence $\psi$ is identically 0.

 $(b)\Longrightarrow (a)$ is obvious.
\end{proof}

Using Corollary~\ref{charcomplilBloch}, we deduce that there are no nontrivial bounded multiplication operators from $H^\infty(\B_n)$ to $\lilBloch(\B_n)$.

\begin{corollary}\label{littlecompactness} The following statements are equivalent:
\begin{enumerate}
\item[(a)] $M_{\psi}:H^\infty(\B_n)\to\lilBloch(\B_n)$ is bounded.
\item[(b)] $M_{\psi}:H^\infty(\B_n)\to\lilBloch(\B_n)$ is compact.
\item[(c)] $\psi$ is identically zero.\end{enumerate}\end{corollary}

We next look at the case when $\psi$ is identically 1, that is, the weighted composition operator reduces to the composition operator from $H^\infty(D)$ into $\Bloch(D)$.

\subsection{Composition Operators from $H^\infty$ into the Bloch Space}

From (\ref{inequality:B_varphi}) we deduce that for any holomorphic self-map of a bounded homogeneous domain $D$, the supremum $\theta_\varphi$ of $\theta_\varphi(z)$, over all $z\in D$, is finite. Indeed, $\theta_\varphi\leq c_DB_\varphi$. Thus, Theorem~\ref{ACbhd} yields the following result.

\begin{corollary}\label{compchar} Let $D$ be a bounded homogeneous domain and let $\varphi$ be a holomorphic self-map of $D$. Then
\begin{enumerate}
\item[(a)] $C_{\varphi}:H^\infty(D)\to\Bloch(D)$ is bounded.
\item[(b)] $C_{\varphi}:H^\infty(D)\to\Bloch_{0*}(D)$ is bounded if and only if $$\lim_{z \to \partial^* D} \theta_\varphi(z) = 0.$$\end{enumerate}
Furthermore, if $C_{\varphi}$ is bounded as an operator into $\Bloch(D)$ or $\Bloch_{0*}(D)$, then
$$1 \leq \norm{C_{\varphi}}\leq 1 + \theta_\varphi.$$
\end{corollary}

\begin{remark} By Corollary~\ref{charcomplilBloch}, all bounded composition operators from $H^\infty(\D)$ to $\lilBloch(\D)$ are compact and the corresponding symbol $\varphi$ must satisfy the condition \begin{equation}\lim_{|z|\to 1}\frac{(1-\mod{z}^2)\mod{\varphi'(z)}}{1-\mod{\varphi(z)}^2}=0.\label{zerolimit}\end{equation}
Besides the symbols whose range is relatively compact in $\D$, examples of symbols satisfying (\ref{zerolimit}) include the functions of the form
$\varphi(z)=\left(\frac{1-\lambda z}2\right)^b$, where $\mod{\lambda}=1$ and $0<b<1$. \end{remark}

\section{Isometries}
\subsection{Isometric multiplication operators} In \cite{AllenColonna:08} and \cite{AllenColonna:08-I}, we proved that the only isometric multiplication operators from the Bloch space of the unit disk or of a bounded symmetric domain that does not have the disk as a factor to itself are those induced by constant functions of modulus one. We now show that there are no isometric multiplication operators from $H^\infty(\D)$ into the Bloch space $\Bloch(\D)$ (and hence into the little Bloch space as well).

\begin{lemma}\label{lemma:no_fix_origin} If $M_\psi:H^\infty(\D) \to \Bloch(\D)$ is an isometry, then $\psi(0)\neq 0$.
\end{lemma}

\begin{proof} Arguing by contradiction, assume $\psi(0) = 0$. Since $M_\psi$ is an isometry, $$\blochnorm{\psi}=\beta_\psi = \blochnorm{M_\psi 1}=1.$$
For $a \in \D$ define the automorphism of $\D$ $$L_a(z) = \frac{a-z}{1-\conj{a}z}, \ z\in\D.$$ Then, $L_a \in H^\infty(\D)$ with $\supnorm{L_a} = 1$.  Again, since $M_\psi$ is an isometry, we obtain $\blochnorm{\psi L_a}=\supnorm{L_a} =1$.  Noting that $$\mod{\psi(a)} = (1-\mod{a}^2)\mod{(\psi L_a)'(a)} \leq \beta_{\psi L_a} =\blochnorm{\psi L_a}= 1,$$ taking the supremum over all $a\in\D$, it follows that $\supnorm{\psi} \leq 1$. Since $\beta_\psi \leq \supnorm{\psi}$, we have $1 = \beta_\psi \leq \supnorm{\psi} \leq 1$.  Thus
\begin{eqnarray} 1 = \supnorm{\psi} = \blochnorm{M_\psi(\psi)} = \beta_{\psi^2}.\label{psisquare1}\end{eqnarray}  On the other hand, by the Schwarz-Pick lemma, we get
$$\begin{aligned}\beta_{\psi^2} &= 2\sup_{z \in \D}(1-\mod{z}^2)\mod{\psi(z)}\mod{\psi'(z)} \leq 2\sup_{z \in \D}\mod{\psi(z)}(1-\mod{\psi(z)}^2)\\
&\leq 2\max_{x \in [0,1]} (x-x^3) = \frac{4}{3\sqrt{3}} < 1,
\end{aligned}$$ which contradicts (\ref{psisquare1}).
\end{proof}

\begin{theorem} There are no isometric multiplication operators $M_\psi$ from $H^\infty(\D)$ to $\Bloch(\D)$.\end{theorem}

\begin{proof} Assume $M_\psi$ is an isometry from $H^\infty(\D)$ into $\Bloch(\D)$.  By Lemma \ref{lemma:no_fix_origin}, the symbol $\psi$ cannot fix the origin. Since the identity function has supremum norm 1, the function $f$ defined by $f(z) = z\psi(z)$ has Bloch semi-norm 1.  Thus, by Theorem~2.1 of \cite{CohenColonna:08}, either $f$ is a rotation or the zeros of $f$ form an infinite sequence $\{z_k\}$ satisfying the condition
\begin{eqnarray}\limsup_{k\to\infty} (1-|z_k|^2)|f'(z_k)|=1.\label{1seminorm}\end{eqnarray} If $f$ is a rotation, then $\psi$ is a constant of modulus 1. Observe that constants cannot induce isometric multiplication operators since there exist functions in $H^\infty(\D)$ with supremum norm 1 which fix the origin and have Bloch semi-norm strictly less than 1 (e.g. the function $z\mapsto z^2$). Thus, $\psi$ cannot be a constant of modulus 1.

If $f$ is not a rotation, then $\psi$ has the same non-zero zeros of $f$ and $|z_k|\to 1$ as $k\to\infty$. Since $f'(z_k)=z_k\psi'(z_k)+\psi(z_k)=z_k\psi'(z_k)$, by (\ref{1seminorm}) we obtain
$$1=\limsup_{k\to\infty}(1-|z_k|^2)|z_k\psi'(z_k)|=\limsup_{k\to \infty}(1-|z_k|^2)|\psi'(z_k)|.$$ Hence $\beta_\psi = 1$. On the other hand, since $\psi$ does not fix the origin, $1 = \blochnorm{\psi} = \mod{\psi(0)} + \beta_\psi > 1$, which yields a contradiction. Therefore, no isometric multiplication operators from $H^\infty(\D)$ into $\Bloch(\D)$ can exist.
\end{proof}

\subsection{Isometric composition operators}

In \cite{Colonna:05} and \cite{AllenColonna:07}, it was shown that there is a large class of isometric composition operators on the Bloch space of the disk and more generally on bounded homogeneous domains that have the disk as a factor. We shall now prove that there are no isometries among the composition operators between the Hardy space $H^\infty$ and the Bloch space of the disk.

\begin{lemma}\label{zerotozero} If $\varphi$ is an analytic self-map of $\D$ inducing an isometric composition operator, then $\varphi(0) = 0$.\end{lemma}

\begin{proof} Assume $C_\varphi$ is an isometry from $H^\infty(\D)$ into $\Bloch(\D)$.  Then letting $f$ be the identity, we have
\begin{eqnarray}\mod{\varphi(0)}+\beta_\varphi= \blochnorm{\varphi} = \blochnorm{C_\varphi f} = \supnorm{f} = 1.\label{phionenorm}\end{eqnarray}
 Furthermore, letting, for $z \in \D$,
$$g_+(z) = \frac{1+z}{2}, \hbox{ and }\, g_-(z) = \frac{1-z}{2},$$ we see that $\supnorm{g_+} = \supnorm{g_-} = 1$.  Thus $\blochnorm{C_\varphi g_+} = \blochnorm{C_\varphi g_-} = 1$, and hence $$\mod{1+\varphi(0)} + \beta_\varphi = 2 = \mod{1-\varphi(0)} + \beta_\varphi.$$
\vskip3pt

\noindent This, combined with (\ref{phionenorm}), yields $\mod{1+\varphi(0)} = 1+\mod{\varphi(0)} = \mod{1-\varphi(0)}$, which implies $\varphi(0) = 0$.
\end{proof}

\begin{theorem} There are no isometric composition operators from $H^\infty(\D)$ to $\Bloch(\D)$.\end{theorem}

\begin{proof}  Suppose $C_\varphi$ is an isometry from $H^\infty(\D)$ to $\Bloch(\D)$.
Consider the function $f(z) = z^2$ for $z \in \D$.  Since $C_\varphi$ is an isometry, we have $\blochnorm{C_\varphi f} = \supnorm{f} = 1$.  On the other hand, arguing as in the proof of Lemma~\ref{lemma:no_fix_origin}, by Lemma~\ref{zerotozero} and the Schwarz-Pick Lemma, we obtain
\begin{eqnarray}
\blochnorm{C_\varphi f} &=& \sup_{z \in \D}2(1-\mod{z}^2)\mod{\varphi(z)}\mod{\varphi'(z)}\nonumber\\
&\leq& \sup_{z \in \D} 2\mod{\varphi(z)}(1-\mod{\varphi(z)}^2) \leq \frac{4}{3\sqrt{3}} < 1,\nonumber
\end{eqnarray}
reaching a contradiction.
\end{proof}

\subsection{Isometric weighted composition operators from $\Bloch$ to $H^\infty$ in the polydisk}
Let $D$ be a bounded homogeneous domain. In \cite{AllenColonna:09}, we characterized the bounded  weighted composition operators from $\Bloch(D)$ to $H^\infty(D)$. In particular, for $D=\D^n$, we obtained a characterization of boundedness equivalent to the following result proved by Li and Stevi\'c in \cite{LiStevic:07}.

\begin{theorem}[\cite{LiStevic:07}] $W_{\psi,\varphi}:\Bloch(\D^n)\to H^\infty(\D^n)$ is bounded if and only if $\psi\in H^\infty(\D^n)$ and \begin{eqnarray}\sup_{z\in D}\mod{\psi(z)}\sum_{k=1}^n\log\frac{2}{1-\mod{\varphi_j(z)}}<\infty.\label{condpolydisk}\end{eqnarray}\end{theorem}

	As a consequence we obtain the following result.

\begin{theorem}\label{noisometries} There are no isometric weighted composition operators from $\Bloch(\D^n)$ to $H^\infty(\D^n)$.
\end{theorem}

\begin{proof} Assume $W_{\psi,\varphi}:\Bloch(\D^n)\to H^\infty(\D^n)$ is an isometry. Then $$\supnorm{\psi}=\supnorm{W_{\psi,\varphi}1}=1$$ and, fixing $j=1,\dots,n$,
$$ \supnorm{\psi\varphi_j}=\supnorm{W_{\psi,\varphi}p_j}=\blochnorm{p_j}=1,$$ where $p_j$ is the projection $z\mapsto z_j$. Thus, there exists a sequence $\{z^{(m)}\}$ in $\D^n$ such that $\mod{\psi(z^{(m)})\varphi_j(z^{(m)})}\to 1$ as $m\to\infty$. Since both $\psi$ and $\varphi_j$ map $\D^n$ into $\overline{\D}$, it follows that $\mod{\psi(z^{(m)})}\to 1$ and $\mod{\varphi_j(z^{(m)})}\to 1$ as $m\to\infty$. On the other hand, by the boundedness of $W_{\psi,\varphi}$, (\ref{condpolydisk}) implies that $\psi(z^{(m)})\to 0$, a contradiction.
\end{proof}

\subsection{Final remarks}
By Corollary~\ref{charcomplilBloch}, the bounded weighted composition operators from $H^\infty(\B_n)$ to $\lilBloch(\B_n)$ are necessarily compact, so there exist no isometric weighted composition operators from $H^\infty(\B_n)$ to $\lilBloch(\B_n)$.
	
We have not been able to prove or disprove the existence of isometric weighted composition operators from $H^\infty$ to the Bloch space, even in the case of the unit disk.

We conjecture that there are no isometric weighted composition operators from $H^\infty(D)$ to $\Bloch(D)$ for any bounded homogeneous domain $D$.

%-------------------------------------------------------------------------------------------------
% Bibliography
%-------------------------------------------------------------------------------------------------
\bibliographystyle{amsplain}
\bibliography{references.bib}
\end{document}